\documentclass[12pt]{article}
\usepackage[T1]{fontenc}
\usepackage[english]{babel}
\usepackage{mathrsfs}
\usepackage{amssymb}	
\usepackage{amsthm}
\usepackage{xcolor}
\usepackage{amsmath}
\usepackage{multirow}
\usepackage{nccmath}
\usepackage{rotating}
\usepackage{subcaption}
\usepackage{hyperref}
\usepackage{multicol}

\usepackage[normalem]{ulem}
		
\usepackage[left = 2cm, right = 2cm, bottom = 2cm, top = 2cm]{geometry}

\newcommand{\weg}[1]{}

\theoremstyle{plain}
\newtheorem{Theorem}{Theorem}

\newtheorem{Definition}{Definition} 

\newtheorem{Proposition}{Proposition}

\usepackage{graphicx}
\graphicspath{ {./images/} }

\title{On three-dimensional $\operatorname{gl}$-regular Nijenhuis operators}
\author{Dinmukhammed Akpan\footnote{Institute of Mathematics, Friedrich Schiller University Jena, 07743 Jena, Germany, Institute of Mathematics and Mathematical Modeling, Almaty, Kazakhstan, dinmukhammed.akpan@uni-jena.de}, Matteo Manenti\footnote{University of Antwerp, Belgium, matteo.manenti@uantwerpen.be}}

\begin{document}
\maketitle
\begin{abstract}
 The paper discusses the classification of three-dimensional gl-regular Nijenhuis operators and their normal forms. We also prove A.~Bolsinov's conjecture.
\end{abstract}


\section*{Introduction}

Let $M$ be a smooth manifold, and let $L$ be an operator field (a tensor field of type $(1,1)$) on it. Then $L$ is called a \emph{Nijenhuis operator} if  
\begin{equation*}
    \mathcal{N}_{L}(\xi,\eta):={L}^2[\xi,\eta]+[{L}\xi,{L}\eta]-{L}[{L}\xi,\eta]-{L}[\xi,{L}\eta] \equiv 0,
\end{equation*}
where $\xi$ and $\eta$ are vector fields and $[\cdot,\cdot]$ denotes their commutator. Note that the expression $\mathcal{N}_{L}(\xi,\eta)$ is called the \emph{Nijenhuis torsion of $L$} and was first introduced by Albert Nijenhuis in \cite{nijenhuis1951xn}. Recently, Nijenhuis operators and their applications have been actively studied in a series of works by A.~Bolsinov, A.~Konyaev, and V.~Matveev \cite{bolsinov2022nijenhuisNGI, konyaev2021nijenhuisNGII, bolsinov2024nijenhuisNGIII, bolsinov2024nijenhuisNGIV}.

Let  
\begin{equation} \label{eq: char pol gen}
    \chi(t)=\det(t\operatorname{Id}-L)=t^n-\sigma_1 t^{n-1}-\cdots-\sigma_{n-1}t-\sigma_n
\end{equation}
be the characteristic polynomial of an operator $L$. In \cite{bolsinov2022nijenhuisNGI} it was proved that if $L$ is a Nijenhuis operator, then in any  coordinate system $x=(x^1,\ldots,x^n)$ the following relation holds:
\begin{equation}\label{ngi form}
    J L = L_{\operatorname{comp1}} J, \quad 
    L_{\operatorname{comp1}} = \begin{pmatrix}
        \sigma_1(x) & 1 & 0 & \ldots & 0 \\
        \sigma_2(x) & 0 & 1 & \ldots & 0 \\
        \ldots & \ldots & \ldots & \ldots & \ldots \\
        \sigma_{n-1}(x) & 0 & 0 & \ldots & 1 \\
        \sigma_n(x) & 0 & 0 & \ldots & 0
    \end{pmatrix},
\end{equation}
where $J$ is the Jacobian matrix of the characteristic map given by the invariants of the operator:
\[
\sigma : M^n \mapsto \mathbb{R}^n, \quad  \
\sigma(x) = (\sigma_1(x),\ldots,\sigma_n(x)).
\]
Thus, if the Jacobian $J$ is invertible almost everywhere, then the Nijenhuis operator $L$ is recovered by the formula
\begin{equation}\label{reconst form}
    L = J^{-1} L_{\operatorname{comp1}} J.
\end{equation}

\begin{Definition}[\cite{bolsinov2024nijenhuisNGIII}]
    An operator $L:\mathbb{R}^n \to \mathbb{R}^n$ is called \textit{$\operatorname{gl}$-regular} if for each eigenvalue of the operator $L$ in its Jordan normal form there is exactly one Jordan block. A Nijenhuis operator $L$ defined on a smooth manifold $M^n$ is called \textit{$\operatorname{gl}$-regular} if it is $\operatorname{gl}$-regular at every point $p \in M^n$.
\end{Definition}

Let $\chi_L(\lambda)$ be the characteristic polynomial of a real-analytic $\operatorname{gl}$-regular operator $L$, given by formula \eqref{eq: char pol gen}. Then the following conditions are equivalent (see \cite{bolsinov2024nijenhuisNGIII}):
\begin{itemize}
    \item $L$ is a Nijenhuis operator;
    \item there exists a coordinate system $x=(x^1,\ldots,x^n)$ in which the operator takes the form $L_{\operatorname{comp1}}$.
    Moreover, the functions $\sigma_i$ satisfy the system of partial differential equations
    \begin{equation}\label{eq: gl-reg}
        \partial_{x^i}
       \left(\sigma^T \right)
        = L_{\operatorname{comp1}} \;
        \partial_{x^{i+1}}
       \left(\sigma^T \right),
        \qquad i=1,\ldots,n-1.
    \end{equation}
\end{itemize}

System \eqref{eq: gl-reg} is compatible and written in Cauchy form. Since all objects are assumed to be everywhere real-analytic, by the Cauchy–Kovalevskaya theorem, for any analytic initial conditions
\[
\sigma_i(0,\ldots,0,x^n)=\sigma_i^0(x^n),\qquad i=1,\ldots,n,
\]
there exists a unique analytic solution.

If $L$ is a $\operatorname{gl}$-regular Nijenhuis operator whose algebraic type does not change in the neighbourhood under consideration, then $L$ has one of the forms presented in Table \ref{table: generic}.

\begin{table}[htbp]
\centering
\renewcommand{\arraystretch}{1.6}
\setlength{\tabcolsep}{8pt}

\begin{tabular}{|p{0.45\textwidth}|p{0.45\textwidth}|}
\hline

\parbox{0.45\textwidth}{
\centering\small
Three distinct real eigenvalues
\[
L=\begin{pmatrix}
\lambda_1(x) & 0 & 0\\
0 & \lambda_2(y) & 0\\
0 & 0 & \lambda_3(z)
\end{pmatrix},
\qquad \lambda_i\neq\lambda_j.
\]
}
&
\parbox{0.45\textwidth}{
\centering\small
One real and two complex conjugate eigenvalues
\[
L=\begin{pmatrix}
a(x,y) & -b(x,y) & 0\\
b(x,y) & a(x,y) & 0\\
0 & 0 & \lambda(z)
\end{pmatrix}
\]
where $a(x,y)+ib(x,y)=\mu(x+iy)$ is holomorphic and $b\neq0$.
}
\\
\hline

\parbox{0.45\textwidth}{
\centering\small Jordan block of size two and one simple eigenvalue
\[
L=\begin{pmatrix}
\lambda(y) & 1 & 0\\
0 & \lambda(y) & 0\\
0 & 0 & \mu(z)
\end{pmatrix},
\qquad \lambda\neq\mu.
\]
}
&
\parbox{0.45\textwidth}{
\centering\small
Jordan cell of size three
\[
L=\begin{pmatrix}
\lambda(x) & 0 & 0\\
1 & \lambda(x) & 0\\
-\lambda' z & 1 & \lambda(x)
\end{pmatrix},
\qquad
\lambda'=\frac{d\lambda(x)}{dx}.
\]
}
\\
\hline

\end{tabular}
\caption{Normal forms for a Nijenhuis operator in a neighbourhood of a generic points.}
\label{table: generic}
\end{table}

We consider the operator in a neighbourhood of a singular point, that is, where the eigenvalues of the operator coincide, and without loss of generality we may assume that they are equal to zero at this point; otherwise we pass to the operator $L \to L - \lambda_0 \operatorname{Id}$ and then assume that this point is the origin.

Note that in \cite{bolsinov2024nijenhuisNGIII} a complete classification of two-dimensional $\operatorname{gl}$-regular Nijenhuis operators in a neighbourhood of a singular point was obtained.

\section*{Main part}

A complete classification of three-dimensional $\operatorname{gl}$-regular Nijenhuis operators is not yet known, but we present some results obtained in this direction.

In what follows, we assume that $(x,y,z)$ is a coordinate system in which the operator is written in the form $L_{\operatorname{comp1}}$ and the characteristic polynomial of a three-dimensional $\operatorname{gl}$-regular Nijenhuis operator has the form
\begin{equation}\label{eq: char pol}
    \chi(t) = t^3 - f(x,y,z)t^2 - g(x,y,z)t - h(x,y,z).
\end{equation}

Moreover, without loss of generality, we consider a neighbourhood of the origin, which is a singular point. In terms of the coefficients of the characteristic polynomial, this condition reads
\[
f(0,0,0)=g(0,0,0)=h(0,0,0)=0.
\]
In our notation, system \eqref{eq: gl-reg} takes the form
\begin{equation}\label{eq:NijenhuisCond}
    \begin{cases}
        \partial_x \begin{pmatrix} f \\ g \\ h \end{pmatrix}
        = L_{\operatorname{comp1}} \; \partial_y \begin{pmatrix} f \\ g \\ h \end{pmatrix},
        \\[2ex]
        \partial_y \begin{pmatrix} f \\ g \\ h \end{pmatrix}
        = L_{\operatorname{comp1}} \; \partial_z \begin{pmatrix} f \\ g \\ h \end{pmatrix}
    \end{cases}
    \quad\Longleftrightarrow\quad
    \begin{cases}
        f_x = f f_y + g_y,\\
        g_x = g f_y + h_y,\\
        h_x = h f_y,\\[2ex]
        f_y = f f_z + g_z,\\
        g_y = g f_z + h_z,\\
        h_y = h f_z.
    \end{cases}
\end{equation}
Denote by
\[
f(0,0,z)=f_0(z), \qquad
g(0,0,z)=g_0(z), \qquad
h(0,0,z)=h_0(z)
\]
the initial data for system \eqref{eq:NijenhuisCond}.

First, consider the cases where one of the invariants of the operator $L$ is identically zero. These are degenerate cases, since the system of invariants becomes functionally dependent everywhere. We present these results without proof, as they follow directly from the analysis of system \eqref{eq:NijenhuisCond}.

Let $J_0$ denote the nilpotent Jordan block, i.e.
\[
J_0 = \begin{pmatrix}
    0 & 1 & 0 & \ldots & 0 \\
    0 & 0 & 1 & \ldots & 0 \\
    \ldots & \ldots & \ldots & \ldots & \ldots \\
    0 & 0 & 0 & \ldots & 1 \\
    0 & 0 & 0 & \ldots & 0
\end{pmatrix}.
\]

\begin{Proposition} \label{st: tr L = 0}
    Let $L$ be a $\operatorname{gl}$-regular Nijenhuis operator (of arbitrary dimension) in a neighbourhood of the origin and suppose $L(\mathbf{0})$ is conjugate to the nilpotent Jordan block $J_0$. Assume $\operatorname{tr} L \equiv 0$. Then all other coefficients of the characteristic polynomial are also identically zero, and the operator $L$ reduces to the form $J_0$ in a whole neighbourhood of the origin.
\end{Proposition}

\begin{Proposition}
    Let $L$ be a real-analytic three-dimensional $\operatorname{gl}$-regular Nijenhuis operator and let $\sigma_2 \equiv 0$ in formula \eqref{eq: char pol gen}. Then $\det L \equiv 0$ and
    \begin{itemize}
        \item either $\operatorname{tr}L \equiv 0$ and the operator reduces to $J_0$;
        \item or, in suitable coordinates $(\tilde{x}, \tilde{y}, \tilde{z})$, the operator $L$ is given by the matrix
        \[
            \begin{pmatrix}
                0 & 1 & 0 \\
                0 & 0 & 1 \\
                0 & 0 & \varepsilon \tilde{z}^d
            \end{pmatrix},
        \]
        for some $d \in \mathbb{N}$ and $\varepsilon = \pm 1$.
    \end{itemize}
\end{Proposition}

\begin{Proposition}
    Let $L$ be a real-analytic three-dimensional $\operatorname{gl}$-regular Nijenhuis operator such that $\det L \equiv 0$, while the other two invariants are functionally independent in a neighbourhood of the origin.
    Then, in suitable coordinates $(\tilde{x}, \tilde{y}, \tilde{z})$, the operator $L$ is given by the matrix
    \[
    \begin{pmatrix}
        \tilde{x} & 1 & 0 \\
        \tilde{y} & 0 & 0 \\
        1 & 0 & 0
    \end{pmatrix}.
    \]
\end{Proposition}

It is known that if a real-analytic $\operatorname{gl}$-regular Nijenhuis operator $L$ satisfies $d(\det L) \neq 0$, then $L$ is a differentially nondegenerate Nijenhuis operator (see definition in \cite{bolsinov2022nijenhuisNGI}). In dimension three, the nondegeneracy condition on the determinant, in terms of the initial data for system \eqref{eq:NijenhuisCond}, reads $h_0'(0)\neq 0$.

Then a natural question arises: what happens if $h_0(z)$ has a zero of high order at $z=0$? In this connection, in \cite{bkm} the following conjecture on the normal form of Nijenhuis operators was presented.

\textbf{Bolsinov's conjecture.}
\textit{Let $L$ be a real-analytic three-dimensional $\operatorname{gl}$-regular Nijenhuis operator given in the form $L_{\operatorname{comp1}}$, and let $f(x,y,z)$, $g(x,y,z)$ and $h(x,y,z)$ be its invariants defined by the characteristic polynomial \eqref{eq: char pol}, with $f(0,0,0) = g(0,0,0) = h(0,0,0) = 0$. 
    Assume that }
    $$
    \begin{cases}
      f(0,0,z) = f_0(z) \quad \text{an arbitrary analytic function,} \\
      g(0,0,z) = g_0(z) = z \hat{g}(z), \ \hat{g}(0) \neq 0, \\
      h(0,0,z) = h_0(z) = z^k \hat{h}(z), \ k \geq 2, \ \hat{h}(0) \neq 0. 
    \end{cases}
    $$
  \textit{Then there exists a coordinate system $(\tilde{x},\tilde{y},\tilde{z})$ in which } $ f = \tilde{x}$, $g = \tilde{z}(\alpha + \tilde{y})$, $ h = \varepsilon \tilde{z}^k$, where $\alpha = \frac{\hat{g}(0)}{|\hat{h}(0)|^{1/k}} \neq 0$ and $\varepsilon = \operatorname{sgn}(\hat{h}(0))$ \textit{and the operator $L$ is given by the matrix}
    \begin{equation}
        L = \begin{pmatrix}
            \tilde{x} & \tilde{z} & \alpha +\tilde{y} \\[2pt]
            \dfrac{k-1}{k}(\alpha +\tilde{y}) & 0 & \varepsilon k\tilde{z}^{k-2} \\[6pt]
            \dfrac{\tilde{z}}{k} & 0 & 0
        \end{pmatrix}.
    \end{equation}

\begin{Theorem}
    Bolsinov's conjecture is true.
\end{Theorem}

\begin{proof}[Sketch of proof.]
From the original system \eqref{eq:NijenhuisCond} it follows that $h = z^k \hat{h}(z) e^{\psi(x,y,z)}$, where $\psi = \psi(x,y,z)$ is a real-analytic function and $\psi(0,0,0) = 0$.

Then we can define a new coordinate $\tilde z := (\varepsilon h)^{1/k} = z \cdot e^{\frac{\psi(x,y,z)}{k}}(\varepsilon \hat{h})^{\frac{1}{k}} = z \cdot H(x,y,z)$, with $H(0,0,0) = |\hat{h}(0)|^{\frac{1}{k}} \neq 0$.

Since $f_y = f f_z + g_z$, and hence $f_y(0,0,0) = g_z(0,0,0) \neq 0$, we take as another new coordinate $\tilde x = f(x,y,z)$.

Now observe that all Taylor coefficients \(g_{ij}(z)\) of the function $g$ are divisible by \(z\), i.e.
\[
g_{ij}(z)=z\hat g_{ij}(z).
\]
This is a technical part of the proof and follows directly from the analysis of system \eqref{eq:NijenhuisCond}. More precisely, by induction one can prove that each derivative of the function \(g\) has the form
\[
i!\,j!\,g_{ij}(z)
=
\frac{\partial^{i+j}g}{\partial x^i\partial y^j}\bigg|_{x=y=0}
=
gP_1^{ij}
+
g_xP_2^{ij}
+
g_yP_3^{ij}
+
Q_{ij}P_4^{ij},
\]
where \(P_1^{ij},P_2^{ij},P_3^{ij},P_4^{ij}\) are polynomials in \(g,f,h\) and their derivatives, and \(Q_{ij}\) is a linear combination of derivatives of the function $h$ of the form $\partial_x^a\partial_y^b\partial_z^c h$,
where \(a,b\ge 1\), \(a+b+c\le i+j\) and \(c=0,1\).
Consequently, each coefficient $g_{ij}(z)$ vanishes at \(z=0\).
Then the function $g$ can be written as
$$
g = g_0(z) + z G(x,y,z) = z\left(\hat{g}(z) + G(x,y,z)\right), \quad \hat{g}(0) \neq 0.
$$
Define the coordinate 
$$\tilde{y} = \frac{g - \alpha z H}{z H} = \frac{\hat{g} + G - \alpha H}{H}.$$  
Note that
$$\frac{\partial \tilde{y}}{\partial x}(0,0,0) = \frac{\hat{g}^2(0)}{|\hat{h}(0)|^{1/k}}\left(1 - \frac{1}{k} \right) = \alpha \hat{g}(0)\left(1 - \frac{1}{k} \right) \neq 0.
$$

It remains to check that $(\tilde{x}, \tilde{y}, \tilde{z})$ give a coordinate system centered at the origin. Indeed,
$$
\begin{cases}
    \tilde{x}(0,0,0) = f(0,0,0) = 0, \\
    \tilde{y}(0,0,0) = \dfrac{\hat{g}(0) - \alpha H(0,0,0)}{H(0,0,0)} = \dfrac{\hat{g}(0) - \frac{\hat{g}(0)}{|\hat{h}(0)|^{1/k}} |\hat{h}(0)|^{1/k}}{ |\hat{h}(0)|^{1/k}} = 0, \\
    \tilde{z}(0,0,0) = (\varepsilon h(0,0,0))^{1/k} = 0,
\end{cases}
$$
and the Jacobian matrix of the change of variables at the origin is
$$
J(0,0,0) = \begin{pmatrix}
    0 & \hat{g}(0) & \star \\
   \alpha \hat{g}(0)\left(1 - \frac{1}{k} \right) & \star & \star \\
    0 & 0 & |\hat{h}(0)|^{1/k}
\end{pmatrix}, \quad |J(0,0,0)| = -\alpha \hat{g}^2(0) |\hat{h}(0)|^{1/k} \left(1 - \frac{1}{k} \right) \neq 0.
$$

In the new coordinates we have $f=\tilde x,
g=\tilde z(\alpha+\tilde y),
h=\varepsilon\tilde z^k.$ Then, from formula \eqref{reconst form}, it follows that in the coordinates $(\tilde x,\tilde y,\tilde z)$ the operator \(L\) takes the required form.
\end{proof}

\section*{Acknowledgements}
The authors express their gratitude to A. Bolsinov for his attention to the text of the work and to Steven Dickerson for the constructed examples. The project was initiated during the workshop "Integrable systems, Liouville foliations and Nijenhuis geometry" at the CIMAT mathematical centre in Guanajuato.

\section*{Funding}
The research was supported by the grant of the Ministry of Science and Higher Education of the Republic of Kazakhstan (grant No. AP23483476) and also by DFG 529233771. M.M. was supported by the FWO-EoS grant through the project \textit{Beyond symplectic geometry} with AU 45816.

\end{document}